\documentclass[11pt,a4paper]{article}
\usepackage{a4}

\usepackage{mathrsfs} %mathscr
\usepackage{amssymb} %Box
\usepackage{amsmath} %Operator
\usepackage{amsthm} %proof
\usepackage{stmaryrd} %bracket
\usepackage{xcolor}
\usepackage{graphicx}
\usepackage{hyperref}

\usepackage{tikz}
\usetikzlibrary{positioning}

\newtheorem{theorem}{Theorem}
\newtheorem{proposition}[theorem]{Proposition}
\newtheorem{lemma}[theorem]{Lemma}

\newtheorem{claim}[theorem]{Claim}

\newcommand{\bigo}[1]{O\left(#1\right)}

\newcommand{\eps}{\varepsilon}
\newcommand{\given}{ \; | \; }

\renewcommand{\theenumi}{{\rm(\roman{enumi})}}

\DeclareMathOperator{\BFS}{BFS}
\DeclareMathOperator{\ex}{ex}

\title{Distance colouring without one cycle length}

\author{
 Ross J. Kang
\thanks{Radboud University Nijmegen, Netherlands. 
Email: {\tt ross.kang@gmail.com}.}
 \and
 Fran\c{c}ois Pirot
\thanks{LORIA, Vand\oe{}uvre-l\`es-Nancy, France; 
Email: {\tt francois.pirot@loria.fr}.}
}

\begin{document}
\maketitle

\begin{abstract}
We consider distance colourings in graphs of maximum degree at most $d$ and how excluding one fixed cycle length $\ell$ affects the number of colours required as $d\to\infty$. For vertex-colouring and $t\ge 1$, if any two distinct vertices connected by a path of at most $t$ edges are required to be coloured differently, then a reduction by a logarithmic (in $d$) factor against the trivial bound $O(d^t)$ can be obtained by excluding an odd cycle length $\ell \ge 3t$ if $t$ is odd or by excluding an even cycle length $\ell \ge 2t+2$. For edge-colouring and $t\ge 2$, if any two distinct edges connected by a path of fewer than $t$ edges are required to be coloured differently, then excluding an even cycle length $\ell \ge 2t$ is sufficient for a logarithmic factor reduction.
For $t\ge 2$, neither of the above statements are possible for other parity combinations of $\ell$ and $t$.
These results can be considered extensions of results due to Johansson (1996) and Mahdian (2000), and are related to open problems of Alon and Mohar (2002) and Kaiser and Kang (2014).

%  Keywords: graph colouring, distance colouring, girth, graph powers, incidence structures.
  
%  AMS 2010 codes: 
%  05C15 (primary), %Coloring of graphs and hypergraphs
%  05C35, %Extremal problems
%  05C70 (secondary). %Factorization, matching, partitioning, covering and packing
\end{abstract}

%%%%%%%%%%%%%%%%%%%%%%%%%%%%%%%%%%%%%%%%%%%%%%%%%%%%%%%%%%%%%%%%%%%%%%

\section{Introduction}\label{sec:intro}

For a positive integer $t$, the {\em $t$-th power} $G^t$ of a (simple) graph $G = (V,E)$ is a graph with vertex set $V$ in which two distinct elements of $V$ are adjacent in $G^t$ if there is a path in $G$ of length at most $t$ between them.
The {\em line graph} $L(G)$ of a graph $G=(V,E)$ is a graph with vertex set $E$ in which two distinct elements are adjacent in $L(G)$ if the corresponding edges of $G$ have a common endpoint.
The {\em distance-$t$ chromatic number} $\chi_t(G)$, respectively, {\em distance-$t$ chromatic index} $\chi'_t(G)$, of $G$ is the chromatic number of $G^t$, respectively, of $(L(G))^t$. 
(So $\chi_1(G)$ is the chromatic number $\chi(G)$ of $G$, $\chi'_1(G)$ the chromatic index $\chi'(G)$ of $G$, and $\chi'_2(G)$ the strong chromatic index $\chi'_s(G)$ of $G$.)

The goal of this work is to address the following basic question.
What is the largest possible value of $\chi_t(G)$ or of $\chi'_t(G)$ among all graphs $G$ with maximum degree at most $d$ that do not contain the cycle $C_\ell$ of length $\ell$ as a subgraph?
For both parameters, we are interested in finding those choices of $\ell$ (depending on $t$) for which there is an upper bound that is $o(d^t)$ as $d\to\infty$.
(Trivially $\chi_t(G)$ and $\chi'_t(G)$ are $O(d^t)$ since the maximum degrees $\Delta(G^t)$ and $\Delta((L(G))^t)$ are $O(d^t)$ as $d\to\infty$.
Moreover, by probabilistic constructions~\cite{AlMo02,KaKa14}, these upper bounds must be $\Omega(d^t/\log d)$ as $d\to \infty$ regardless of the choice of $\ell$.)
We first discuss some previous work.

For $t=1$ and $\ell=3$, the question for $\chi_t$ essentially was a long-standing problem of Vizing~\cite{Viz68}, one that provoked much work on the chromatic number of triangle-free graphs, and was eventually settled asymptotically by Johansson~\cite{Joh96}. He used nibble methods to show that the largest chromatic number over all triangle-free graphs of maximum degree at most $d$ is $\Theta(d/\log d)$ as $d\to \infty$. It was observed in~\cite{KaPi16} that this last statement with $C_\ell$-free, $\ell>3$, rather than triangle-free also holds, thus completely settling this question asymptotically for $\chi_1=\chi$.

Regarding the question for $\chi'_t$, first notice that since the chromatic index of a graph of maximum degree $d$ is either $d$ or $d+1$, there is little else to say asymptotically if $t=1$.

For $t=2$ and $\ell=4$, the question for $\chi'_t$ was considered by Mahdian~\cite{Mah00} who showed that the largest strong chromatic chromatic index over all $C_4$-free graphs of maximum degree at most $d$ is $\Theta(d^2/\log d)$ as $d\to \infty$.
Vu~\cite{Vu02} extended this to hold for any fixed bipartite graph instead of $C_4$, which in particular implies the statement for any $C_\ell$, $\ell$ even. Since the complete bipartite graph $K_{d,d}$ satisfies $\chi'_2(K_{d,d}) = d^2$, the statement does not hold for $C_\ell$, $\ell$ odd.
This completely settles the second question asymptotically for $\chi'_2=\chi'_s$.

In this paper, we advance a systematic treatment of our basic question.
Our main results are as follows, which may be considered as extensions of the results of Johansson~\cite{Joh96} and Mahdian~\cite{Mah00} to distance-$t$ vertex- and edge-colouring, respectively, for all $t$.

\begin{theorem}%[Even cycle forbidden. Both $\chi_t$ and $\chi'_t$.]
\label{thm:even}
Let $t$ be a positive integer and $\ell$ an even positive integer.
\begin{enumerate}
\item\label{thm:even,vertex}
For $\ell \ge 2t+2$, the supremum of the distance-$t$ chromatic number over $C_\ell$-free graphs of maximum degree at most $d$ is $\Theta(d^t/\log d)$ as $d\to \infty$.
\item\label{thm:even,edge}
For $t\ge 2$ and $\ell \ge 2t$, the supremum of the distance-$t$ chromatic index over $C_\ell$-free graphs of maximum degree at most $d$ is $\Theta(d^t/\log d)$ as $d\to \infty$.
\end{enumerate}
\end{theorem}

\begin{theorem}%[Odd cycle forbidden. Just $\chi_t$.]
\label{thm:odd}
Let $t$ and $\ell$ be odd positive integers such that $\ell \ge 3t$.
The supremum of the distance-$t$ chromatic number over $C_\ell$-free graphs of maximum degree at most $d$ is $\Theta(d^t/\log d)$ as $d\to \infty$.
\end{theorem}

This study was initiated by a conjecture of ours in~\cite{KaPi16}, that the largest distance-$t$ chromatic number over all $C_{2t+2}$-free graphs of maximum degree at most $d$ is $\Theta(d^t/\log d)$ as $d\to \infty$.
Theorem~\ref{thm:even}\ref{thm:even,vertex} confirms our conjecture.

In Section~\ref{sec:constructions}, we exhibit constructions to certify the following, so improved upper bounds are impossible for the parity combinations of $t$ and $\ell$ other than those in Theorems~\ref{thm:even} and~\ref{thm:odd}.

\begin{proposition}\label{prop:constructions}
Let $t$ and $\ell$ be positive integers.
\begin{enumerate}
\item\label{prop:constructions,odd}
For $t$ even and $\ell$ odd, the supremum of the distance-$t$ chromatic number over $C_\ell$-free graphs of maximum degree at most $d$ is $\Theta(d^t)$ as $d\to \infty$.
\item\label{prop:constructions,odd,index}
For $t\ge 2$ and $\ell$ odd, the supremum of the distance-$t$ chromatic index over $C_\ell$-free graphs of maximum degree at most $d$ is $\Theta(d^t)$ as $d\to \infty$.
\end{enumerate}
\end{proposition}

We have reason to suspect that the values $2t+2$ and $2t$, respectively, may not be improved to lower values in Theorem~\ref{thm:even}, but we do not go so far yet as to conjecture this.
We also wonder whether the value $3t$ in Theorem~\ref{thm:odd} is optimal --- it might well only be a coincidence for $t=1$ --- but we know that in general it may not be lower than $t$, as we show in Section~\ref{sec:constructions}.

Our basic question in fact constitutes refined versions of problems of Alon and Mohar~\cite{AlMo02} and of Kaiser and the first author~\cite{KaKa14}, which instead asked about the asymptotically extremal distance-$t$ chromatic number and index, respectively, over graphs of maximum degree $d$ and girth at least $g$ as $d\to\infty$. Our upper bounds imply bounds given earlier in~\cite{AlMo02,KaKa14,KaPi16}, and the lower bound constructions given there are naturally relevant here (as we shall see in Section~\ref{sec:constructions}).

It is worth pointing out that the basic question unrestricted, i.e.~asking for the extremal value of the distance-$t$ chromatic number or index over graphs of maximum degree $d$ as $d\to\infty$, is likely to be very difficult if we ask for the precise (asymptotic) multiplicative constant. This is because the question for $\chi_t$ then amounts to a slightly weaker version of a well-known conjecture of Bollob\'as on the degree--diameter problem~\cite{Bol78}, while the question for $\chi'_t$ then includes the notorious strong edge-colouring conjecture of Erd\H{o}s and Ne\v{s}et\v{r}il, cf.~\cite{Erd88}, as a special case.

Our proofs of Theorems~\ref{thm:even} and~\ref{thm:odd} rely on direct applications of the following result of Alon, Krivelevich and Sudakov~\cite{AKS99}, which bounds the chromatic number of a graph with bounded neighbourhood density.

\begin{lemma}[\cite{AKS99}]\label{lem:AKS}
For all graphs $\hat{G} = (\hat{V},\hat{E})$ with maximum degree at most $\hat{\Delta}$ such that for each $\hat{v}\in \hat{V}$ there are at most $\frac{\hat{\Delta}^2}{f}$ edges spanning $N(\hat{v})$, it holds that $\chi(\hat{G}) = \bigo{\frac{\hat{\Delta}}{\log f}}$ as $\hat{\Delta}\to\infty$.
\end{lemma}

\noindent
The proof of this result in~\cite{AKS99} invoked Johannson's result for triangle-free graphs; using nibble methods directly instead, Vu~\cite{Vu02} extended it to hold for list colouring. So Theorems~\ref{thm:even} and~\ref{thm:odd} also hold with list versions of $\chi_t$ and $\chi'_t$.

Section~\ref{sec:proofs} is devoted to showing the requisite density properties for Lemma~\ref{lem:AKS}.
In order to do so with respect to Theorem~\ref{thm:even}, we in part use a classic result of Bondy and Simonovits~\cite{BoSi74} that the Tur\'an number $\ex(n,C_{2k})$ of the even cycle $C_{2k}$, that is, the maximum number of edges in a graph on $n$ vertices not containing $C_{2k}$ as a subgraph, satisfies $\ex(n,C_{2k})=O(n^{1+1/k})$ as $n\to\infty$. We also use a technical refinement which we describe and prove in Section~\ref{sec:proofs}.

We made little effort to optimise the multiplicative constants implicit in Theorems~\ref{thm:even} and~\ref{thm:odd} and in Proposition~\ref{prop:constructions}. Importantly, the constants we obtained depend on $\ell$ or $t$, and it is left to future work to determine the nature of the true precise dependencies.

\section{Constructions}\label{sec:constructions}
In this section, we describe some constructions that certify the conclusions of Theorems~\ref{thm:even} and~\ref{thm:odd} are not possible with other parity combinations of $t$ and $\ell$, in particular showing Proposition~\ref{prop:constructions}.

First we review constructions we used in previous work~\cite{KaPi16}.
In combination with the trivial bound $\chi_t(G) =O(d^t)$ if $\Delta(G)\le d$, the following two propositions imply Proposition~\ref{prop:constructions}\ref{prop:constructions,odd}.
The next result also shows that the value $3t$ in Theorem~\ref{thm:odd} may not be reduced below $t$.

\begin{proposition}\label{lem:construction}
Fix $t \ge 3$.
For every even $d \ge 2$, there exists a $d$-regular graph $G$ with $\chi_t(G) \ge d^t/2^t$ and $\chi'_{t+1}(G) \ge d^{t+1}/2^t$. Moreover, $G$ is bipartite if $t$ is even, and $G$ does not contain any odd cycle of length less than $t$ if $t$ is odd.
\end{proposition}

\begin{proof}

We define $G = (V,E)$ as follows.
The vertex set is $V = \cup_{i=0}^{t-1}U^{(i)}$ where each $U^{(i)}$ is a copy of $[d/2]^t$, the set of ordered $t$-tuples of symbols from $[d/2]=\{1,\dots,d/2\}$.
For all $i\in\{0,\dots,t-1\}$, we join an element $(x^{(i)}_0,\dots,x^{(i)}_{t-1})$ of $U^{(i)}$ and an element $(x^{(i+1 \bmod t)}_0,\dots,x^{(i+1 \bmod t)}_{t-1})$ of $U^{(i+1 \bmod t)}$ by an edge if the $t$-tuples agree on all symbols except possibly at coordinate $i$,
i.e.~if $x^{(i+1 \bmod t)}_j = x^{(i)}_j$ for all $j\in\{0,\dots,t-1\}\setminus\{i\}$ (and $x^{(i)}_i$, $x^{(i+1 \bmod t)}_i$ are arbitrary from $[d/2]$).

It is easy to see that each $U^{(i)}$ is a clique in $G^t$, and every set of edges incident to some $U^{(i)}$ is a clique in $(L(G))^{t+1}$. This gives $\chi_t(G) \ge |U^{(0)}| = (d/2)^t$ and $\chi'_{t+1}(G) \ge d\cdot|U^{(0)}| = 2(d/2)^{t+1}$.
(In fact here it is easy to find a colouring achieving equality in both cases.)

Since $G$ is composed only of bipartite graphs arranged in sequence around a cycle of length $t$, every odd cycle in $G$ is of length at least $t$, and $G$ is bipartite if $t$ is even.
\end{proof}

As observed in~\cite{AlMo02} and~\cite{KaKa14}, certain finite geometries yield bipartite graphs of prescribed girth giving better bounds than in Proposition~\ref{lem:construction} for a few cases.

\begin{proposition}\label{prop:geometries}
Let $d$ be one more than a prime power.
\begin{itemize}
\item There exists a bipartite, girth $6$, $d$-regular graph ${\cal P}_{d-1}$ with $\chi_2({\cal P}_{d-1}) = d^2-d+1$ and $\chi'_3({\cal P}_{d-1}) = d^3-d^2+d$. 
\item There exists a bipartite, girth $8$, $d$-regular graph ${\cal Q}_{d-1}$ with $\chi'_4({\cal Q}_{d-1}) = d^4-2d^3+2d^2$. 
\item There exists a bipartite, girth $12$, $d$-regular graph ${\cal H}_{d-1}$ with $\chi'_6({\cal H}_{d-1}) = d^6-4d^5+7d^4-6d^3+3d^2$. 
\end{itemize}
\end{proposition}

\begin{proof}
Letting ${\cal P}_{d-1}$ be the point-line incidence graph of the projective plane $PG(2,d-1)$,
${\cal Q}_{d-1}$ that of a symplectic quadrangle with parameters $(d-1,d-1)$, and
${\cal H}_{d-1}$ that of a split Cayley hexagon with parameters $(d-1,d-1)$,
it is straightforward to check that these graphs satisfy the promised properties.
\end{proof}

In~\cite{KaPi16}, we somehow combined Propositions~\ref{lem:construction} and~\ref{prop:geometries} for other lower bound constructions having prescribed girth. This approach is built upon generalised $n$-gons, structures which are known not to exist for $n>8$~\cite{FeHi64}.
We refer the reader to~\cite{KaPi16} for further details.

\medskip
Our second objective in this section is to introduce a different graph product applicable only to two regular balanced bipartite graphs. We use it to produce two bipartite constructions for $\chi'_t$, both of which settle the case of $t$ even left open in Proposition~\ref{lem:construction}, and the second of which treats what could be interpreted as an edge version of the degree--diameter problem.

Let $H_1 = (V_1=A_1\cup B_1,E_1)$ and $H_2 = (V_2=A_2\cup B_2,E_2)$ be two balanced bipartite graphs with given vertex orderings, i.e.~$A_1 = \{a^1_1,\dots,a^1_{n_1}\}$, $B_1 = \{b^1_1,\dots,b^1_{n_1}\}$, $A_2 = \{a^2_1,\dots,a^2_{n_2}\}$, $B_2 = \{b^2_1,\dots,b^2_{n_2}\}$ for some positive integers $n_1$, $n_2$.
We define the {\em balanced bipartite product $H_1\bowtie H_2$ of $H_1$ and $H_2$} as the graph with vertex and edge sets defined as follows:
\begin{align*}
V_{H_1\bowtie H_2} & := A_1\times A_2 \cup B_1\times B_2 \text{ and }\\
E_{H_1\bowtie H_2} & := \{(a^1_i,a^2)(b^1_i,b^2) | i\in\{1,\dots,n_1\},a^2b^2\in E_2\} \cup \\
& \qquad\qquad\{(a^1,a^2_j)(b^1,b^2_j) | a^1b^1\in E_1,j\in\{1,\dots,n_2\}\}.
\end{align*}
See Figure~\ref{fig:product} for an example of this product.

\begin{figure}
\centering
\resizebox{0.7\textwidth}{!}{%!TEX root = distcycle.tex
\begin{tikzpicture}[-,>=,node distance=0.8cm,scale=1,draw,nodes={circle,draw,fill=black, inner sep=1.5pt}]

\node (centering1) [fill=none,draw=none] {};
\node (centering2) [right of=centering1,fill=none,draw=none] {};

\node (b11) [above of=centering1] {};
\node (a11) [below of=centering1] {};
\node (b12) [above of=centering2] {};
\node (a12) [below of=centering2] {};

\node (b11label) [above=0.1 of b11,fill=none,draw=none] {$b^1_1$};
\node (a11label) [below=0.1 of a11,fill=none,draw=none] {$a^1_1$};
\node (b12label) [above=0.1 of b12,fill=none,draw=none] {$b^1_2$};
\node (a12label) [below=0.1 of a12,fill=none,draw=none] {$a^1_2$};

\node (bowtie) [right of=centering2,fill=none,draw=none] {\huge $\bowtie$};
\node (centering3) [right of=bowtie,fill=none,draw=none] {};
\node (centering4) [right of=centering3,fill=none,draw=none] {};

\node (b21) [above of=centering3] {};
\node (a21) [below of=centering3] {};
\node (b22) [above of=centering4] {};
\node (a22) [below of=centering4] {};

  \path (a11) edge [very thick] node[fill=none,draw=none] {} (b11);
  \path (a12) edge [very thick] node[fill=none,draw=none] {} (b11);
  \path (a11) edge [very thick] node[fill=none,draw=none] {} (b12);
  \path (a12) edge [very thick] node[fill=none,draw=none] {} (b12);

  \path (a21) edge [very thick] node[fill=none,draw=none] {} (b21);
  \path (a22) edge [very thick] node[fill=none,draw=none] {} (b21);
  \path (a21) edge [very thick] node[fill=none,draw=none] {} (b22);
  \path (a22) edge [very thick] node[fill=none,draw=none] {} (b22);

\node (b21label) [above=0.1 of b21,fill=none,draw=none] {$b^2_1$};
\node (a21label) [below=0.1 of a21,fill=none,draw=none] {$a^2_1$};
\node (b22label) [above=0.1 of b22,fill=none,draw=none] {$b^2_2$};
\node (a22label) [below=0.1 of a22,fill=none,draw=none] {$a^2_2$};

\node (equals) [right=0.8 of centering4,fill=none,draw=none] {\huge $=$};
\node (centering5) [right=0.8 of equals,fill=none,draw=none] {};
\node (centering6) [right=1.3 of centering5,fill=none,draw=none] {};
\node (centering7) [right=1.3 of centering6,fill=none,draw=none] {};
\node (centering8) [right=1.3 of centering7,fill=none,draw=none] {};

\node (bb11) [above=1.5 of centering5] {};
\node (aa11) [below=1.5 of centering5] {};
\node (bb12) [above=1.5 of centering6] {};
\node (aa12) [below=1.5 of centering6] {};
\node (bb21) [above=1.5 of centering7] {};
\node (aa21) [below=1.5 of centering7] {};
\node (bb22) [above=1.5 of centering8] {};
\node (aa22) [below=1.5 of centering8] {};

  \path (aa11) edge [very thick] node[fill=none,draw=none] {} (bb11);
  \path (aa11) edge [very thick] node[fill=none,draw=none] {} (bb12);
  \path (aa11) edge [very thick] node[fill=none,draw=none] {} (bb21);
  \path (aa12) edge [very thick] node[fill=none,draw=none] {} (bb11);
  \path (aa12) edge [very thick] node[fill=none,draw=none] {} (bb12);
  \path (aa12) edge [very thick] node[fill=none,draw=none] {} (bb22);
  \path (aa21) edge [very thick] node[fill=none,draw=none] {} (bb11);
  \path (aa21) edge [very thick] node[fill=none,draw=none] {} (bb21);
  \path (aa21) edge [very thick] node[fill=none,draw=none] {} (bb22);
  \path (aa22) edge [very thick] node[fill=none,draw=none] {} (bb12);
  \path (aa22) edge [very thick] node[fill=none,draw=none] {} (bb21);
  \path (aa22) edge [very thick] node[fill=none,draw=none] {} (bb22);

\node (bb11label) [above=-0.2 of bb11,fill=none,draw=none] {$(b^1_1,b^2_1)$};
\node (aa11label) [below=-0.2 of aa11,fill=none,draw=none] {$(a^1_1,a^2_1)$};
\node (bb12label) [above=-0.2 of bb12,fill=none,draw=none] {$(b^1_1,b^2_2)$};
\node (aa12label) [below=-0.2 of aa12,fill=none,draw=none] {$(a^1_1,a^2_2)$};
\node (bb21label) [above=-0.2 of bb21,fill=none,draw=none] {$(b^1_2,b^2_1)$};
\node (aa21label) [below=-0.2 of aa21,fill=none,draw=none] {$(a^1_2,a^2_1)$};
\node (bb22label) [above=-0.2 of bb22,fill=none,draw=none] {$(b^1_2,b^2_2)$};
\node (aa22label) [below=-0.2 of aa22,fill=none,draw=none] {$(a^1_2,a^2_2)$};
\end{tikzpicture}}
\caption{An illustration of the product $K_{2,2}\bowtie K_{2,2}$.\label{fig:product}}
\end{figure}
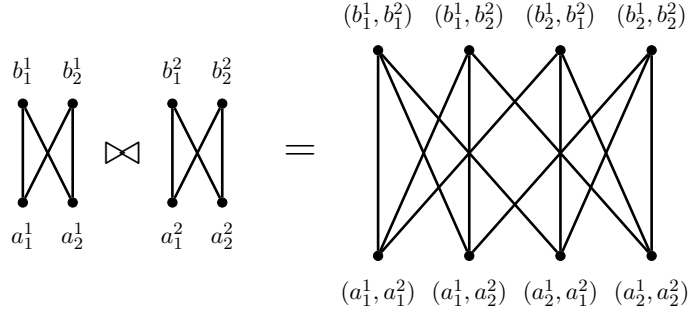

Usually the given vertex orderings will be of either of the following types.
We say that a labelling $A = \{a_1,\dots,a_n\}$, $B = \{b_1,\dots,b_n\}$ of $H=(V=A\cup B, E)$ is a {\em matching ordering} of $H$ if $a_ib_i\in E$ for all $i\in\{1,\dots,n\}$.
We say it is a {\em comatching ordering} if $a_ib_i\notin E$ for all $i\in\{1,\dots,n\}$.
Note by Hall's theorem that every non-empty regular balanced bipartite graph admits a matching ordering, while every non-complete one admits a comatching ordering.

Let us now give some properties of this product relevant to our problem, especially concerning its degree and distance properties. The first of these propositions follow easily from the definition.

\begin{proposition}\label{prop:productdegree}
Let $H_1$ and $H_2$ be two balanced bipartite graphs that have part sizes $n_1$ and $n_2$, respectively, and are regular of degrees $d_1$ and $d_2$, respectively, for some positive integers $n_1,n_2,d_1,d_2$. 
Suppose $H_1$, $H_2$ are given in either matching or comatching ordering.
Then $H_1\bowtie H_2$ is a regular balanced bipartite graph with parts $A_{H_1\bowtie H_2} = A_1\times A_2$ and $B_{H_1\bowtie H_2} = B_1\times B_2$ each of size $n_1n_2$. If both are in matching ordering, then $H_1\bowtie H_2$ has degree $d_1+d_2-1$, otherwise it has degree $d_1+d_2$.
\end{proposition}

\begin{proposition}\label{prop:productdistance}
Let $H_1 = (V_1=A_1\cup B_1,E_1)$ and $H_2 = (V_2=A_2\cup B_2,E_2)$ be two regular balanced bipartite graphs.
\begin{enumerate}
\item\label{prop:productdistance,eveneven}
Suppose that for every $a^1,a'^1\in X_1\subseteq A_1$ there is a $t_1$-path between $a^1$ and $a'^1$ in $H_1$ (for some $t_1$ even).
Suppose that for every $a^2,a'^2\in X_2\subseteq A_2$ there is a $t_2$-path between $a^2$ and $a'^2$ in $H_2$ (for some $t_2$ even).
Then for every $(a^1,a^2),(a'^1,a'^2)\in X_1\times X_2 \subseteq A_{H_1\bowtie H_2}$, there is a $(t_1+t_2)$-path between $(a^1,a^2)$ and $(a'^1,a'^2)$ in $H_1\bowtie H_2$.
\item\label{prop:productdistance,evenodd}
Suppose that for every $a^1,a'^1\in X_1\subseteq A_1$ there is a $t_1$-path between $a^1$ and $a'^1$ in $H_1$ (for some $t_1$ even).
Suppose that for every $a^2\in X_2\subseteq A_2$ and $b^2\in Y_2\subseteq B_2$ there is a $t_2$-path between $a^2$ and $b^2$ in $H_2$ (for some $t_2$ odd).
Then for every $(a^1,a^2)\in X_1\times X_2\subseteq A_{H_1\bowtie H_2}$ and $(b^1,b^2)\in Y_1\times Y_2 \subseteq B_{H_1\bowtie H_2}$ where $Y_1 = \{b^1_i\given a^1_i\in X_1\}$, there is a $(t_1+t_2)$-path between $(a^1,a^2)$ and $(b^1,b^2)$ in $H_1\bowtie H_2$.
\end{enumerate}
\end{proposition}

\begin{proof}
We only show part~\ref{prop:productdistance,evenodd}, as the other part is established in the same manner.
Let $(a^1,a^2) \in X_1\times X_2$ and $(b^1,b^2)\in Y_1\times Y_2$.
Using the distance assumption on $H_1$, let $a^1_{i_0},b^1_{i_1},a^1_{i_2},\cdots, b^1_{i_{t_1-1}},a^1_{i_{t_1}}$ be a $t_1$-path in $H_1$ between $a^1=a^1_{i_0}$ and $a^1_{i_{t_1}}$, where $i_{t_1}$ is such that $b^1 = b^1_{i_{t_1}}$.
Using the distance assumption on $H_2$, let $a^2_{j_0}b^2_{j_1}a^2_{j_2}\cdots a^2_{j_{t_2-1}}b^2_{j_{t_2}}$ be a $t_2$-path in $H_2$ between $a^2=a^2_{j_0}$ and $b^2=b^2_{j_{t_2}}$.
The following $(t_1+t_2)$-path between $(a^1,a^2)$ and $(b^1,b^2)$ in $H_1\bowtie H_2$ traverses using one of the coordinates, then the other:
\begin{align*}
(a^1,a^2)=
(a^1_{i_0},a^2_{j_0})(b^1_{i_1},b^2_{j_0})(a^1_{i_2},a^2_{j_0})\cdots (b^1_{i_{t_1-1}},b^2_{j_0})(a^1_{i_{t_1}},a^2_{j_0})&\\
(b^1_{i_{t_1}},b^2_{j_1})(a^1_{i_{t_1}},a^2_{j_2})\cdots (a^1_{i_{t_1}},b^2_{j_{t_2-1}})(b^1_{i_{t_1}},b^2_{j_{t_2}})&
=(b^1,b^2).\qedhere
\end{align*}
\end{proof}

We use this product to show that no version of Theorem~\ref{thm:odd} may hold for $\chi'_t$.
In combination with the trivial bound $\chi'_t(G)=O(d^t)$ if $\Delta(G)\le d$, we deduce Proposition~\ref{prop:constructions}\ref{prop:constructions,odd,index} from Proposition~\ref{lem:construction}, the following result and the fact that $\chi'_2(K_{d,d})=d^2$.

\begin{proposition}
Fix $t \ge 4$ even. For every $d\ge 2$ with  $d \equiv 0 \pmod{2(t-2)}$, there exists a $d$-regular bipartite graph $G$ with $\chi'_t(G) \ge d^t/(et2^{t-1})$.
\end{proposition}

\begin{proof}
Let $t_1 = t-2$ and $d_1 = (t_1-1)d/t_1$. Let $G_1=(V_1,E_1)$ be the construction promised by Proposition~\ref{lem:construction} for $d_1$ and $t_1$. Since $t_1$ is even, we can write $V_1 = A_1\cup A_2$ where $A_1 = \cup \{U^{(i)} \given i\in\{0,\dots,t-1\}\text{ even}\}$ and $B_1 = \cup \{U^{(i)} \given i\in\{0,\dots,t-1\}\text{ odd}\}$.
This is a $d_1$-regular balanced bipartite graph, and for every $a_1,a'_1 \in U^{(0)}\subseteq A_1$ there exists a $t_1$-path between $a_1$ and $a'_1$.
Moreover, it is possible to label $A_1$ and $B_1$ in comatching ordering so that the indices for $U^{(i)}$ coincide with those for $U^{(i+1)}$ for every $i\in\{0,2,\dots,t-2\}$.

Let $t_2 = 1$ and $d_2=d-d_1=d/t_1$.
Let $G_2 = (V_2 =A_2\cup B_2,E_2) = K_{d_2,d_2}$. This is a $d_2$-regular balanced bipartite graph, and for every $a_2\in A_2,b_2\in B_2$, there exists a $t_2$-path between $a_2$ and $b_2$. Trivially any labelling of $A_2$ and $B_2$ gives rise to a matching ordering.

Let $G = G_1\bowtie G_2$, $X = U^{(0)}\times A_2$ and $Y = U^{(1)}\times B_2$.
Now $G$ is a $d$-regular bipartite graph by Proposition~\ref{prop:productdegree}, and by Proposition~\ref{prop:productdistance} for every $(a_1,a_2)\in X$ and $(b_1,b_2)\in Y$, there exists a $(t-1)$-path between $(a_1,a_2)$ and $(b_1,b_2)$. Thus the edges of $G$ that span $X\times Y$ induce a clique in $(L(G))^t$. The number of such edges is (since $t>3$) at least
\[
\left(\frac{d_1}{2}\right)^{t_1}d_2\left(\frac{d_1}{2} + d_2\right) = \left(1-\frac{1}{t-2}\right)^{t-2}\frac{(t-1)d^t}{(t-3)^22^{t-1}} \ge \frac{d^t}{et2^{t-1}}.
\qedhere
\]
\end{proof}

Alternatively, Proposition~\ref{prop:constructions}\ref{prop:constructions,odd,index} follows from the following result, albeit at the expense of a worse dependency on $t$ in the multiplicative factor. For $t\ge 2$, we can take a $(t-1)$-th power of the product operation to produce a bipartite graph $G$ of maximum degree $d$ with $\Omega(d^t)$ edges such that $(L(G))^t$ is a clique.

\begin{proposition}
\label{prop:edge-diameter}
Fix $t\geq 2$. For every $d\ge 2$ with $d \equiv 1 \pmod{t-1}$, there exists a $d$-regular bipartite graph $G=(V,E)$ with $|E| =d\cdot ((d-1)/(t-1)+1)^{t-1}$ and $\chi'_t(G) = |E|$. 
\end{proposition}

\begin{proof}
Let $d' = (d-1)/(t-1)+1$ and $G = \bowtie^{t-1} K_{d',d'}$, the $(t-1)$-th power of $K_{d',d'}$ under the product $\bowtie$, where the factors are always taken in matching ordering. By Proposition~\ref{prop:productdegree}, $G$ is a $d$-regular bipartite graph and has $d\cdot d'^{t-1}$ edges. By Proposition~\ref{prop:productdistance}, there is a path of length at most $t-1$ between every pair of vertices in the same part if $t-1$ is even, or in different parts if $t-1$ is odd. It follows that $(L(G))^t$ is a clique.
\end{proof}

\section{Proofs of Theorems~\ref{thm:even} and~\ref{thm:odd}}\label{sec:proofs}

In this section we prove the main theorems. Before proceeding, let us set notation and make some preliminary remarks.

Let $G=(V,E)$ be a graph.
We will often need to specify the vertices at some fixed distance from a vertex or an edge of $G$.
Let $i$ be a non-negative integer.
If $x\in V$, we write $A_i = A_i(x)$ for the set of vertices at distance exactly $i$ from $x$.
If $e\in E$, we write $A_i = A_i(e)$ for the set of vertices at distance exactly $i$ from an endpoint of $e$.
We shall often abuse this notation by writing $A_{\le j}$ for $\cup_{i\le j}A_i$ and so forth.
We will write $G_i = G[A_i,A_{i+1}]$ to be the bipartite subgraph induced by the sets $A_i$ and $A_{i+1}$

We will also often need to specify a unique breadth-first search subgraph $\BFS=\BFS(x)$ ($\BFS=\BFS(e)$, respectively) rooted at $x$ ($e$, respectively). Having fixed an ordering of $V$ beforehand, i.e.~writing $V=\{1,\dots,|V|\}$, $\BFS$ is a graph on $V$ whose edges are defined as follows. For every $v\in A_i$, $i > 0$, we include the edge to that neighbour of $v$ in $A_{i-1}$ being least in the ordering.

In proving the distance-$t$ chromatic number upper bounds in Theorems~\ref{thm:even} and~\ref{thm:odd} using Lemma~\ref{lem:AKS}, given $x\in V$, we need to consider the number of pairs of distinct vertices in $A_{\le t}$ that are connected by a path of length at most $t$. It will suffice to prove that this number is $O(d^{2t-\eps})$ as $d\to\infty$ for some fixed $\eps>0$.
In fact, in our enumeration we may restrict our attention to paths of length {\em exactly} $t$ 
whose endpoints are in $A_t$. This is because $|A_i| \le d^i$ for all $i$ and the number of paths of length exactly $j$ containing some fixed vertex is at most $(j+1)d^j$ for all $j$.

Similarly, in proving the distance-$t$ chromatic index upper bound in Theorem~\ref{thm:even} using Lemma~\ref{lem:AKS}, given $e\in E$, we need to consider the number of pairs of distinct edges that each have at least one endpoint in $A_{<t}$ and that are connected by a path of length at most $t-1$. It will suffice to prove that this number is $O(d^{2t-\eps})$ as $d\to\infty$ for some fixed $\eps>0$.
Similarly as above, in our enumeration we may restrict our attention to paths of length {\em exactly} $t-1$
whose endpoint edges both intersect $A_{t-1}$.

As mentioned in the introduction, for Theorem~\ref{thm:even} we will show a technical refinement of the classic bound on $\ex(n,C_{2k})$ of Bondy and Simonovits~\cite{BoSi74}. We borrow heavily from the strategy used by Pikhurko~\cite{Pik12} to obtain the following improvement on the bound of Bondy and Simonovits: for all $k\ge 2$ and $n\ge 1$,
\begin{align}\label{eqn:pikhurko}
\ex(n, C_{2k}) \le (k-1)n^{1+1/k} + 16(k-1)n.
\end{align}
This has since been improved by Bukh and Jiang~\cite{BuJi17}.
We will require a bound like~\eqref{eqn:pikhurko} that only counts edges of a certain type in bipartite $C_{2k}$-free graphs, but depends on the cardinality of only one of the parts.
Given a bipartite graph $H=(V=A\cup B,E)$, we call an edge {\em $\delta$-bunched with respect to $A$} if it is incident to a vertex in $B$ of degree at least $\delta$.

\begin{lemma}\label{lem:turanbipartite}
Fix $k\ge 2$. For any $n_A\ge 1$, if $H=(V=A\cup B,E)$ is a bipartite $C_{2k}$-free graph with $|A|=n_A$, then the number of edges that are $\delta$-bunched with respect to $A$ is at most $\delta n_A$, where $\delta = 2(k-1)n_A^{1/k} + 16(k-1)$.
\end{lemma}

\begin{proof}
For a contradiction, let us assume for some $n_A\ge 1$ that there exists a $C_{2k}$-free bipartite graph $H=(V=A\cup B,E)$ with $|A|=n_A$ such that the number of $\delta$-bunched edges with respect to $A$ is more than $\delta n_A$. Let $H'=(V'=A'\cup B',E')$ be the bipartite subgraph of $H$ induced by those $\delta$-bunched edges so that $A'\subseteq A$ and $B'\subseteq B$ have smallest cardinality.

Every vertex in $B'$ has degree at least $\delta$, and the vertices in $A'$ have degree at least $|E'|/|A'| > \delta$ by assumption. So the average degree $d(H')$ of $H'$ is more than $\delta$. This implies that $H'$ contains a subgraph $H''=(V''=A''\cup B'',E'')$ of minimum degree at least $\delta'$ where $\delta' = d(H')/2>\delta/2 = (k-1)n_A^{1/k} + 8(k-1)$, and so that $A''\subseteq A'$ and $B''\subseteq B'$.

Let $x\in A''$ if $k$ is odd and $x\in B''$ otherwise.
For every $i\ge 0$, we define $V_i$ to be the set of vertices at distance $i$ from $x$ in $H''$, and $H_i = H[V_i,V_{i+1}]$ to be the bipartite subgraph of $H''$ induced by the sets $V_i$ and $V_{i+1}$. We use two intermediary results from~\cite{Pik12} concerning the presence of a {\em $\Theta$-subgraph}, defined to be any subgraph that is a cycle of length at least $2k$ with a chord:

\begin{lemma}[\cite{Pik12}]
Let $k\ge 3$. Any bipartite graph of minimum degree at least $k$ contains a $\Theta$-subgraph.
\end{lemma}

\begin{claim}[\cite{Pik12}]
For $i\in\{0,\dots, k-1\}$, $H_i$ contains no $\Theta$-subgraph. 
\end{claim}

These two statements, together with the fact that there is always a subgraph whose minimum degree is at least half the average degree of the (super)graph, imply that for $k\ge 3$ and $i\in\{0,\dots,k-1\}$, the average degree of $H_i$ must satisfy $d(H_i) \le 2k-2$.
Note that since $H''$ is bipartite there are no edges of $H''$ spanning $V_i$ for any $i\in\{1,\dots,k-1\}$.

We now show by induction for every $i\in\{0,\dots,k-1\}$ that the average degree from $V_{i+1}$ to $V_i$ is at most $k-1+\eps$ where $\eps = 3(k-1)^2/\delta'$, i.e.~that the number $e(H_i)$ of edges in $H_i$ satisfies
\begin{align}\label{eqn:upper}
e(H_i) \le (k-1+\eps)|V_{i+1}|.
\end{align}
The base case $i=0$ is clearly true since each vertex of $V_1$ has exactly one edge to $V_0$.
Now for the induction let $i\in\{1,\dots,k-1\}$ and assume that the statement is true for $i-1$.
By the inductive hypothesis and the properties of $H''$, we have
\begin{align}\label{eqn:lower}
e(H_i) = \sum_{y\in V_i}\deg_{V_{i+1}}(y) \ge (\delta'-(k-1+\eps))|V_i|.
\end{align}
This shows that the average degree from $V_i$ to $V_{i+1}$ is at least $\delta'-(k-1+\eps)\ge 2k-2$ (where we used that $\delta'\ge8(k-1)$ and $\eps \le 3(k-1)/8$). In particular, $|V_{i+1}|>0$.
It cannot be that the average degree from $V_{i+1}$ to $V_i$ is greater than $2k-2$, or else we would not have $d(H_i) \le 2k-2$. So $e(H_i) \le (2k-2)|V_{i+1}|$. Combining this with~\eqref{eqn:lower}, we obtain
\begin{align*}
|V_i| \le \frac{2k-2}{\delta'-(k-1+\eps)}|V_{i+1}|.
\end{align*}
Again using the density condition on $H_i$, this implies
\begin{align*}
\frac{2e(H_i)}{(1+\frac{2k-2}{\delta'-(k-1+\eps)})|V_{i+1}|}\le\frac{2e(H_i)}{|V_{i+1}|+|V_i|}=d(H_i) \le 2k-2
\end{align*}
which in turn implies (using that $\delta'\ge8(k-1)$ and $\eps \le 3(k-1)/8$) that
\begin{align*}
e(H_i) \le \left(k-1+\frac{2(k-1)^2}{\delta'-(k-1+\eps)}\right)|V_{i+1}| \le (k-1+\eps)|V_{i+1}|,
\end{align*}
as required.

Combining~\eqref{eqn:upper} and~\eqref{eqn:lower}, we have that for all $i\in\{0,\dots,k-1\}$
\begin{align*}
\frac{|V_{i+1}|}{|V_i|} \ge \frac{\delta'-(k-1+\eps)}{k-1+\eps} \ge \frac{\delta'}{k-1+2\eps},
\end{align*}
where the last inequality again uses that $\delta'\ge8(k-1)$ and $\eps \le 3(k-1)/8$ together with the definition of $\eps$. Since $\delta'>(k-1)n_A^{1/k} + 8(k-1) \ge (k-1)n_A^{1/k}$, we have by the choice of $x$ in $A''$ or $B''$ that
\begin{align*}
n_A \ge |A''| \ge |V_k| \ge \left(\frac{\delta'}{k-1+2\eps}\right)^k > \left(\frac{(k-1)n_A^{1/k} + 8(k-1)}{k-1+6(k-1)n_A^{-1/k}}\right)^k > n_A,
\end{align*}
a contradiction. This completes the proof.
\end{proof}

\begin{proof}[Proof of Theorem~\ref{thm:even}\ref{thm:even,vertex}]
By the probabilistic construction described in~\cite{AlMo02}, it suffices to prove only the upper bound in the statement. We may also assume that $t\geq 2$, since it was already observed in~\cite{KaPi16} that for any $\ell \geq 3$ the chromatic number of any $C_\ell$-free graph of maximum degree $d$ is $O(d/\log d)$. 

Let $\ell \ge 2t+2$ be even, let $G=(V,E)$ be a graph of maximum degree at most $d$ such that $G$ contains no $C_\ell$ as a subgraph, and let $x\in V$.
Let $T$ denote the number of pairs of distinct vertices in $A_t$ that are connected by a path of length exactly $t$.
Let $\eps = (\ell-2t)/\ell \in (0,1)$.
As discussed at the beginning of the section, it suffices for the proof to show that $T\le Cd^{2t-\eps}$ where $C$ is a constant independent of $d$, by Lemma~\ref{lem:AKS}.

Let us count the possibilities for a path $x_0\dots x_t$ of length $t$ between two distinct vertices $x_0,x_t\in A_t$. 
Setting $\delta=2(\ell/2-1)d^{2t/\ell} + 16(\ell/2-1)$, we consider three cases.
\begin{enumerate}
\item
The penultimate vertex in the path satisfies $x_{t-1}\in A_{t-1}\cup A_t$. By the assumption that $t\geq 2$, we obtain that $|A_{t-1}\cup A_t| \leq d^t$. By~\eqref{eqn:pikhurko}, the number of edges in $G[A_{t-1}\cup A_t]$ is at most $(\ell/2-1)d^{t(1+2/\ell)}+16(\ell/2-1)d^t$. So the number of choices for $x_{t-1}x_t$ is at most $8.5\ell d^{t+1-\eps}$.
%, and even $\ell/2 d^{t+1-\eps}$ when $d \geq (8\ell)^\frac{\ell}{2t}$. 
The number of choices for the rest of the path is at most $d^{t-1}$.

\item The penultimate vertex satisfies $x_{t-1}\in A_{t+1}$ and $x_{t-1}x_t$ is $\delta$-bunched in $G_t$ with respect to $A_t$. Then Lemma~\ref{lem:turanbipartite} ensures that there are at most $\delta d^t = 2(\ell/2-1)d^{t(1+2/\ell)} + 16(\ell/2-1)d^t \le 9\ell d^{t+1-\eps}$ choices for $x_{t-1}x_t$.
%, or even $\ell d^{t+1-\eps}$ when $d\geq (4\ell)^\frac{\ell}{2t}$. 
The number of choices for the rest of the path is at most $d^{t-1}$.

\item The penultimate vertex satisfies $x_{t-1}\in A_{t+1}$ and $x_{t-1}x_t$ is not $\delta$-bunched in $G_t$ with respect to $A_t$. By the definition of $\delta$-bunched, there are fewer than $\delta \le 9\ell d^{1-\eps}$ choices for $x_t$ given $x_{t-1}$, and so at most $9\ell d^{t-\eps}$ choices given $x_0$.
%, or even at most $\ell d^{t-\eps}$ choices when $d \geq (4\ell)^\frac{\ell}{2t}$. 
There are at most $d^t$ choices for $x_0$.
\end{enumerate}
Summing over the above cases, the overall number of choices for the path $x_0\dots x_t$ is at most $26.5\ell d^{2t-\eps}$, 
%which can be reduced to $2.5\ell d^{2t-\eps}$ when $d \geq (8\ell)^\frac{\ell}{2t}$, 
giving the required bound on $T$.
\end{proof}

\begin{proof}[Proof of Theorem~\ref{thm:even}\ref{thm:even,edge}]
By the probabilistic construction described in~\cite{KaKa14}, it suffices to prove only the upper bound in the statement. To that end, let $\ell \ge 2t$ be even, let $G=(V,E)$ be a graph of maximum degree at most $d$ such that $G$ contains no $C_\ell$ as a subgraph, and let $e\in E$.
Let $T$ denote the number of pairs of distinct edges in $G[A_{t-1}]$ or $G_{t-1}$ that are connected by a path of length $t-1$.
Let $\eps = (\ell-2t+2)/\ell \in (0,1)$.
As discussed at the beginning of the section, it suffices to show that $T\le Cd^{2t-\eps}$ where $C$ is a constant independent of $d$, by Lemma~\ref{lem:AKS}.

Let us count the possibilities for a path $x_0\dots x_{t+1}$ where $x_1\dots x_t$ is a path of length $t-1$ between two distinct edges $x_0x_1$ and $x_tx_{t+1}$ of $G[A_{t-1}]$ or $G_{t-1}$. 
Setting
$\delta=2(\ell/2-1)(2d^{t-1})^{2/\ell} + 16(\ell/2-1)$,
 we consider five cases.
\begin{enumerate}
\item The last edge $x_tx_{t+1}$ is $\delta$-bunched in $G_{t-1}$ with respect to $A_{t-1}$. Then using Lemma~\ref{lem:turanbipartite} there are at most $2\delta d^{t-1} = 2^{2+2/\ell}(\ell/2-1)d^{(t-1)(1+2/\ell)} + 32(\ell/2-1)d^{t-1} \le (16+2^{1+2/\ell})\ell d^{t-\eps}$ choices for $x_tx_{t+1}$.
%, and even at most $2^{1+2/\ell}\ell d^{t-\eps}$ when $d\geq (8\ell)^\frac{\ell}{2(t-1)}$. 
The number of choices for the rest of the path is at most $d^t$.

\item The penultimate vertex satisfies $x_t\in A_t$ and $x_tx_{t+1}$ is not $\delta$-bunched in $G_{t-1}$ with respect to $A_{t-1}$. By the definition of $\delta$-bunched, there are fewer than $\delta \le (8+2^{2/\ell})\ell d^{1-\eps}$ choices for $x_{t+1}$ given $x_t$.
%, or even fewer than $2^{2/\ell}\ell d^{1-\eps}$ when $d\geq (8\ell)^\frac{\ell}{2(t-1)}$.
The number of choices for $x_0x_1$ is at most $4d^t$ (where the extra factor $2$ accounts for whether $x_0$ or $x_1$ is in $A_t$) and then given $x_1$ there are at most $d^{t-1}$ choices for $x_1\dots x_t$. So there are at most $4(8+2^{2/\ell})\ell d^{2t-\eps}$ choices for this case.
%, and even at most $2^{2+2/\ell}\ell d^{2t-\eps}$ choices when $d\geq (8\ell)^\frac{\ell}{2(t-1)}$. 

\item The penultimate vertex satisfies $x_t\in A_{t-1}$ and the penultimate edge $x_{t-1}x_t$ is contained in $ G[A_{t-2}\cup A_{t-1}]$. By (\ref{eqn:pikhurko}) together with the fact that  $|A_{t-2}\cup A_{t-1}|\leq 2d^{t-1}$, the number of edges in $G[A_{t-2}\cup A_{t-1}]$ is at most $2^{1+2/\ell}(\ell/2-1)d^{(t-1)(1+2/\ell)} + 32(\ell/2-1)d^{t-1}$. So the number of choices for $x_{t-1}x_t$ is at most $(16+2^{2/\ell})\ell d^{t-\eps}$.
%, and even $2^{2/\ell}\ell d^{t-\eps}$ when $d\geq (8\ell)^\frac{\ell}{2(t-1)}$.
The number of choices for the rest of the path is at most $d^t$.

\item The penultimate vertex satisfies $x_t\in A_{t-1}$ and the penultimate edge $x_{t-1}x_t$ is $\delta$-bunched in $G_{t-1}$ with respect to $A_{t-1}$. Again by Lemma~\ref{lem:turanbipartite} there are at most $2\delta d^{t-1} \le (16+2^{1+2/\ell})\ell d^{t-\eps}$ choices for $x_{t-1}x_t$.
%, and even at most $2^{1+2/\ell}\ell d^{t-\eps}$ when $d\geq (8\ell)^\frac{\ell}{2(t-1)}$. 
The number of choices for the rest of the path is at most $d^t$.

\item The penultimate vertex satisfies $x_t\in A_{t-1}$ and $x_{t-1}x_t$ is not $\delta$-bunched in $G_{t-1}$ with respect to $A_{t-1}$. By the definition of $\delta$-bunched, there are fewer than $\delta \le (8+2^{2/\ell})\ell d^{1-\eps}$ choices for $x_{t-1}$ given $x_t$.
%, and even at most $2^{2/\ell}$ choices when $d\geq (8\ell)^\frac{\ell}{2(t-1)}$. 
The number of choices for $x_{t+1}x_t$ is at most $4d^t$. The number of choices for the rest of the path is at most $d^{t-1}$, so there are at most $4(8+2^{2/\ell})\ell d^{2t-\eps}$ choices for this case.
%, and even at most $2^{2+2/\ell}\ell d^{2t-\eps}$ choices when $d\geq (8\ell)^\frac{\ell}{2(t-1)}$.
\end{enumerate}
Summing over the above cases, the overall number of choices for the path $x_0\dots x_{t+1}$ is at most $(112+12\cdot 2^{2/\ell})\ell d^{2t-\eps}$
%, or even at most $32\time 2^{2/\ell}\ell$ when $d\geq (8\ell)^\frac{\ell}{2(t-1)}$, 
giving the required bound on $T$.
\end{proof}

\begin{proof}[Proof of Theorem~\ref{thm:odd}]
By the probabilistic construction described in~\cite{AlMo02}, it suffices to prove only the upper bound in the statement. 
Moreover, we may assume $t\ge 3$ due to Johansson's result~\cite{Joh96} and our observation in~\cite{KaPi16}.

Let $\ell \ge 3t$ be odd, let $G=(V,E)$ be a graph of maximum degree at most $d$ such that $G$ contains no $C_\ell$ as a subgraph, and let $x\in V$.
Any path all of whose non-endpoint vertices are in $A_{\ge t}$ we call {\em peripheral}.
Let $T$ denote the number of pairs of distinct vertices in $A_t$ that are connected by a peripheral path of length $t$.
For the same reasons as discussed at the beginning of the section, it suffices for the proof to show that $T\le Cd^{2t-1}$ where $C$ is a constant independent of $d$, by Lemma~\ref{lem:AKS}.

Since $\ell$ is odd, we may write $\ell = 3t+2k$ for some non-negative integer $k$.
For $j\in\{0,1,\dots,2k\}$, let us call a vertex $v\in A_t$ {\em $j$-implantable} if it is the endpoint of some peripheral path of length $j$, the other endpoint of which (if it exists) is in $A_t\setminus\{v\}$.

Fix $v$ to be a $2k$-implantable vertex and $P=v_0v_1\dots v_{2k}$ a path certifying its implantability, so that $v_0=v$ and (if $k>0$) $v_{2k}\in A_t\setminus\{v\}$.
Notice that crudely the number of peripheral paths of length $t$ starting at $v$ which intersect $P$ at another vertex is at most $2ktd^{t-1}$.
Now consider the set $Y \subseteq A_t\setminus \{v\}$ such that there is a peripheral path of length $t$ between $v$ and $y$ that does not intersect $P$ except at $v$ for all $y\in Y$.
Let us see that there is some vertex $a_Y\in A_1$ such that $a_Y$ is an ancestor of every element of $Y$ in the tree $\BFS$.
If not, then there exist distinct $y_1,y_2\in Y$ such that the lowest common ancestor of $y_1$ and $y_2$ in $\BFS$ is $x$. Since $\BFS$ is a tree,  the lowest common ancestor of $v_{2k}$ and $y_1$ (without loss of generality) is also $x$. But then we have found a cycle of length $3t+2k$ that contains $x$, $v_{2k}$, $v$, $y_1$, in that order, a contradiction. Thus $|Y| \le d^{t-1}$, the number of pairs with $v$ that are counted by $T$ is at most $(1+2kt)d^{t-1}$, and the number of pairs with a $2k$-implantable vertex that are counted by $T$ is at most $(1+2kt)d^{2t-1}$.

Observe we are already done if $k=0$, so assume from here on that $k>0$.

It remains for us to (crudely) count the number of paths $z_0\dots z_t$ of length $t$ that have two distinct non-$2k$-implantable endpoints $z_0,z_t \in A_t$.
Let $\kappa_0 = t$ if $k \bmod{t} = 0$ and $\kappa_0 = k \bmod{t}$ otherwise. We organise our count according to a parameter $j < 2k$ defined as the largest integer such that $j\ge 2\kappa_0$, $j \equiv 2\kappa_0 \pmod{t}$, and one of $z_0$ and $z_t$ is $j$-implantable. If this is not defined, then we know that $z_0$ and $z_t$ are not $2\kappa_0$-implantable so we may let $j=0$.

If $j=0$, then trivially the number of choices for $z_0$ is at most $d^t$ and the number of choices for the sub-path $z_0\dots z_{t-\kappa_0}$ is $d^{t-\kappa_0}$. Given $z_{t-\kappa_0}$, the choice for the remainder subpath $z_{t-\kappa_0}\dots z_t$ is restricted by the fact that $z_t$ may not be $2\kappa_0$-implantable; in particular, all such sub-paths must intersect at a vertex other than $z_{t-\kappa_0}$. It follows that the number of choices for $z_0\dots z_t$ in this case is at most $d^t\cdot d^{t-\kappa_0} \cdot \kappa_0^2d^{\kappa_0-1} = \kappa_0^2d^{2t-1}$.

So suppose that $j > 0$. By the definition of $j$, we deem one of $z_0$ and $z_t$ to be $j$-implantable.
Fix $v$ to be a $j$-implantable vertex and $P=v_0v_1\dots v_j$ a path certifying its implantability, so that $v_0=v$ and $v_j\in A_t\setminus\{v\}$. The number of peripheral paths of length $t$ starting at $v$ which intersect $P$ is at most $jtd^{t-1}$.
Consider the set $Y \subseteq A_t\setminus \{v\}$ such that there is a peripheral path of length $t$ between $v$ and $y$ that does not intersect $P$ except at $v$ for all $y\in Y$. Every $y\in Y$ is $(j+t)$-implantable, as certified by the path $P$ concatenated with the path certifying $y\in Y$.
Recalling the definition of $j$, in particular that neither $z_0$ nor $z_t$ may be $(j+t)$-implantable, we have just shown that the number of choices for $z_0\dots z_t$ in this case is at most $2jtd^{2t-1}$.

By summing over all possible $j$, we obtain that the overall number of choices for $z_0\dots z_t$ is at most
$\left(\kappa_0^2 + 2\sum_{\iota=0}^{(2k-2\kappa_0)/t} (2\kappa_0 + \iota t)t\right) d^{2t-1}$.
It therefore follows that $T\le \left(1+2kt+\kappa_0^2 + 2\sum_{\iota=0}^{(2k-2\kappa_0)/t} (2\kappa_0 + \iota t)t\right)d^{2t-1}$, as required.
\end{proof}

Our impression is that it might be possible to improve upon the value $3t$ in Theorem~\ref{thm:odd}; however, in order to do so, it seems one would have to take a different approach.
This is because of a simple construction of a $d$-regular graph $G$ with no odd cycle of length less than $3t$ such that $G^t$ does not satisfy the density conditions demanded by Lemma~\ref{lem:AKS}.
Roughly, we take the main example of Proposition~\ref{lem:construction} but around a cycle of length $3t$ rather than of length $t$.
More precisely, the vertex set is $\cup_{i=0}^{3t-1}U^{(i)}$ where each $U^{(i)}$ is a copy of $[d/2]^t$.
For all $i\in\{0,\dots,3t-1\}$, we join an element $(x^{(i)}_0,\dots,x^{(i)}_{t-1})$ of $U^{(i)}$ and an element $(x^{(i+1 \bmod 3t)}_0,\dots,x^{(i+1 \bmod 3t)}_{t-1})$ of $U^{(i+1 \bmod 3t)}$ by an edge if the $t$-tuples agree on all symbols except possibly at coordinate $i \bmod t$.
It is straightforward to check that $G^t$ is a graph in which all vertices have degree $\Theta(d^t)$ and every neighbourhood is spanned by $\Theta(d^{2t})$ edges, meaning that Lemma~\ref{lem:AKS} is ineffective here.
But neither is $G$ an example to certify sharpness of the value $3t$ in Theorem~\ref{thm:odd}, since it is also straightforward to check that $\chi_t(G) =o(d^t)$.

\section{Concluding remarks and open problems}\label{sec:conclusion}

Our goal was to address the question, what is the asymptotically largest value of $\chi_t(G)$ or of $\chi'_t(G)$ among graphs $G$ with maximum degree at most $d$ containing no cycle of length $\ell$, where $d\to\infty$?
The case $t=1$ for both parameters and the case $t=2$ for $\chi'_t$ followed from earlier work, but we showed more generally that for each fixed $t$ this question for both parameters can be settled apart from a finite number of cases of $\ell$.
These exceptional cases are a source of mystery. We would be very interested to learn if the cycle length constraints $2t$, $2t+2$ and $3t$ in Theorems~\ref{thm:even} and~\ref{thm:odd} could be weakened (or not).

More specifically, writing $\chi_t(d,\ell) = \sup\{\chi_t(G) \given \Delta(G)\le d, G \supsetneq C_\ell\}$ and $\chi'_t(d,\ell) = \sup\{\chi'_t(G) \given \Delta(G)\le d, G \supsetneq C_\ell\}$, the following questions are natural, even if there is no manifest monotonicity in $\ell$.
\begin{enumerate}
\renewcommand{\theenumi}{{\rm\arabic{enumi}.}}
\item For each $t\ge 1$, is there a critical even $\ell^{\text{e}}_t$ such that for any even $\ell$, if $\ell < \ell^{\text{e}}_t$ then $\chi_t(d,\ell) = \Theta(d^t)$, while if $\ell \ge \ell^{\text{e}}_t$ then $\chi_t(d,\ell) = \Theta(d^t/\log d)$?
\item For each $t\ge 2$, is there a critical even $\ell'_t$ such that for any even $\ell$, if $\ell < \ell'_t$ then $\chi'_t(d,\ell) = \Theta(d^t)$, while if $\ell \ge \ell'_t$ then $\chi'_t(d,\ell) = \Theta(d^t/\log d)$?
\item For each $t\ge 1$ odd, is there a critical odd $\ell^{\text{o}}_t$ such that for any odd $\ell$, if $\ell < \ell^{\text{o}}_t$ then $\chi_t(d,\ell) = \Theta(d^t)$, while if $\ell \ge \ell^{\text{o}}_t$ then $\chi_t(d,\ell) = \Theta(d^t/\log d)$?
\end{enumerate}
We knew from before that $\ell^{\text{e}}_1=4$, $\ell^{\text{o}}_1=3$, $\ell^{\text{e}}_2=6$, $\ell'_2=4$, $\ell'_3=6$, $\ell'_4=8$, and $\ell'_6=12$. In this paper, we showed that there are linear in $t$ upper bounds on all these critical values, provided the values are well-defined.

The above three questions are natural analogues to open questions of Alon and Mohar~\cite{AlMo02} and of Kaiser and the first author~\cite{KaKa14} that ask for a critical girth $g_t$ (resp.~$g'_t$) for which there is an analogous decrease in the asymptotic extremal behaviour of the distance-$t$ chromatic number (resp.~index).
If these critical values all exist, it would be natural to think that $g_t = \min\{\ell^{\text{e}}_t,\ell^{\text{o}}_t\}$ and $g'_t=\ell'_t$, and moreover, if $t$ is odd, that $|\ell^{\text{o}}_t-\ell^{\text{e}}_t|=1$. But there is limited evidence for the existence questions, let alone this stronger set of assertions.
We have already established other lower bounds for these hypothetical critical values in~\cite{KaPi16}, but for none of these critical values is there any general construction known to certify a lower bound that is unbounded as $t\to\infty$ .

As mentioned in the introduction, Vu~\cite{Vu02} proved that the exclusion of any fixed bipartite graph is sufficient for a $O(d^2/\log d)$ upper bound on the strong chromatic index of graphs of maximum degree $d$. One might wonder, similarly, for each $t\ge 2$ is there a natural wider class of graphs than sufficiently large cycles (of appropriate parity) whose exclusion leads to asymptotically better upper bounds on the distance-$t$ chromatic number or index?

\bibliographystyle{abbrv}
\bibliography{distcycle}
\end{document}